\documentclass[11pt]{amsart}

\usepackage{amssymb}
\usepackage{latexsym}
\usepackage{amsmath}
\usepackage{amsthm}
\usepackage{amsfonts}
\usepackage{color}
\usepackage{pdfsync}
\usepackage[usenames,dvipsnames]{pstricks}
\usepackage{epsfig}
\usepackage{pst-grad} 
\usepackage{pst-plot} 
\usepackage{enumitem}
\usepackage{nicefrac}

\newcommand{\R}{\mathbb{R}}

\newcommand{\T}{\mathbb{T}}
\newcommand{\Tr}{\mathrm{Tr}}

\theoremstyle{plain}
\newtheorem{defi}{Definition}[section]
\newtheorem{prop}[defi]{Proposition}
\newtheorem{teo}[defi]{Theorem}
\newtheorem{cor}[defi]{Corollary}

\newtheorem{remark}[defi]{Remark}

\theoremstyle{definition}

\theoremstyle{remark}

\numberwithin{equation}{section}

\begin{document}

\title[]{Periodic homogenization of the principal eigenvalue of second-order elliptic operators}

\author[]{Gonzalo D\'avila}
\address{Gonzalo D\'avila: Universidad T\'ecnica Federico Santa Mar\'ia, Departamento de Matem\'atica, Avenida Espa\~na 1680, Valpara\'iso, Chile}
\email{gonzalo.davila@usm.cl}

\author[]{Andrei Rodr\'iguez-Paredes}
\address{Andrei Rodr\'iguez-Paredes: Departamento de Matem\'atica, Universidad de Concepci\'on, Avenida Esteban Iturra s/n - Barrio Universitario - Casilla 160 C - Concepci\'on, Chile}
\email{andreirodriguez@udec.cl}

\author[]{Erwin Topp}
\address{
Erwin Topp:
Departamento de Matem\'atica y C.C., Universidad de Santiago de Chile,
Casilla 307, Santiago, Chile.}
\email{erwin.topp@usach.cl}

\date{\today}

\begin{abstract}
In this paper we investigate homogenization results for the principal eigenvalue problem associated to $1$-homogeneous, uniformly elliptic, second-order operators. Under rather general assumptions, we prove that the principal eigenpair associated to an oscillatory operator converges to the eigenpair associated to the effective one. This includes the case of fully nonlinear operators. Rates of convergence for the eigenvalues are provided for linear and nonlinear problems, under extra regularity/convexity assumptions. Finally, a linear rate of convergence (in terms of the oscillation parameter) of suitably normalized eigenfunctions is obtained for linear problems.
\end{abstract}

\maketitle

\section{Introduction.}\label{intro}

Let $\Omega \subset \R^N$ be a bounded domain with smooth boundary, and consider a non-divergence form, linear  elliptic operator with the form
\begin{equation}\label{linintro}
L = a^{ij}(x) \partial_{x_i} \partial_{x_j} + b^j(x) \partial_{x_j}  + c(x), \quad x \in \Omega,
\end{equation}
such that $a^{ij} \in C(\Omega)$, $b^j, c \in L^\infty(\Omega)$, satisfying the uniform ellipticity condition
\begin{equation}\label{ellintro}
\lambda |\xi|^2 \leq a^{ij}(x) \xi_i \xi_j \leq \Lambda |\xi|^2, \quad x \in \Omega, \ \xi \in \R^N,
\end{equation}
for some given constants $0 < \lambda \leq \Lambda$. In the pioneering work of Berestycki, Nirenberg and Varadhan~\cite{BNV}, the authors provide a complete study of the principal eigenvalue problem associated to $L$ and $\Omega$, namely, the existence of a pair $(\phi_1, \lambda_1) \in C_+(\Omega) \times \R$ solving (in an adequate weak sense) the eigenvalue problem
\begin{equation*}
\left \{ \begin{array}{rll} L \phi_1 & = -\lambda_1 \phi_1 \quad & \mbox{in} \ \Omega \\ \phi_1 & = 0 \quad & \mbox{on} \ \partial \Omega. \end{array} \right .
\end{equation*} 

Here we have denoted by $C_+(\Omega)$ the space of continuous functions which are positive in $\Omega$. For equations in divergence form, the principal eigenvalue coincides with the first eigenvalue in the sense of the classical Rayleigh-Ritz formula. In particular, it is \textsl{unique} (in the sense that it is the unique eigenvalue associated to an eigenfunction that is positive in $\Omega$), and \emph{simple} (in the sense that the associated eigenspace is one-dimensional).

One of the remarkable features of the methods presented in~\cite{BNV} is the possibility to extend the study of the principal eigenvalue problem for fully nonlinear, uniformly elliptic operators with the form
\begin{equation}\label{Isaacs}
F = \inf_{\alpha \in \mathcal A} \sup_{\beta \in \mathcal B} L_{\alpha, \beta},
\end{equation}
where, given sets of indices $\mathcal A, \mathcal B$, $\{ L_{\alpha, \beta} \}_{\alpha \in \mathcal A, \beta \in \mathcal B}$ is a family of linear operators with the form~\eqref{linintro}. 
This has been studied, for instance, by Armstrong~\cite{A} and Quaas-Sirakov~\cite{QS}: under suitable assumptions on $F$, 
the authors prove the existence of a pair $(\phi_1, \lambda_1) \in C_+(\Omega) \times \R$, solving, in the viscosity sense, the Dirichlet problem
\begin{equation}\label{deflambdaFOmega}
\left \{ \begin{array}{rll} F(x, \phi_1, D\phi_1, D^2 \phi_1) & = -\lambda_1 \phi_1, \quad & \mbox{in} \ \Omega \\ \phi_1 & = 0 \quad & \mbox{on} \ \partial \Omega. \end{array} \right .
\end{equation} 

In particular, the principal eigenvalue $\lambda_1 = \lambda_1^+(F, \Omega)$ is characterized as
\begin{equation}\label{caract}
\lambda_1^+(F, \Omega) = \sup \{ \lambda :  \exists \ \phi > 0 \ \mbox{in} \ \Omega, \ F(x, \phi, D\phi, D^2\phi) \leq  -\lambda \phi \ \mbox{in} \ \Omega \},
\end{equation}
where the inequality is understood in the viscosity sense. This eigenvalue is unique and simple (the latter meaning that the eigenfunctions are unique up to a \textsl{positive} multiplicative constant). It is possible to consider the eigenvalue problem associated to \textit{negative} eigenfunctions, but its analysis is somewhat analogous to that of its positive counterpart (at least in what respects our interest here) and therefore we concentrate in the positive eigenpair.

\medskip

The purpose of this article is the study of stability results for the principal eigenvalue problem in the context of periodic homogenization. We introduce the main assumptions in order to present our results.
Denote by $\T^N$ the flat $N$-dimensional torus, and by $S^N$ the space of $N \times N$ symmetric matrices. We consider $F \in C(\Omega \times \T^N \times \R \times \R^N \times S^N)$, constants $0 < \lambda \leq \Lambda < +\infty$ and $C_1,C_2 > 0$ such that
    \begin{equation}\label{elliptic} 
    \begin{split}
        &M_{\lambda, \Lambda}^- (Y) - C_1(|q| + |s|)-C_2(|z| + |w|) \\
\leq & \ F(x+z, y+w, r + s, p + q, X + Y) - F(x,y,r,p,X) \\
 \leq & \ M_{\lambda, \Lambda}^+(Y) + C_1(|q| + |s|)+C_2(|z| + |w|),
        \end{split}
    \end{equation}
for all $x, z \in \Omega, y, w \in \T^N$, $r, s\in \R, p, q \in \R^N$ and $X, Y \in S^N$.    
Here, $M_{\lambda, \Lambda}^{\pm}$ denote the extremal Pucci operators (see, e.g., ~\cite{CC})
$$
M^+_{\lambda, \Lambda} (X) = \sup_{\lambda I \leq A \leq \Lambda I} \mathrm{Tr}(AX), \quad  M^-_{\lambda, \Lambda} (X) = \inf_{\lambda I \leq A \leq \Lambda I} \mathrm{Tr}(AX).
$$ 

We also assume $F$ is positively $1$-homogeneous, that is 
    \begin{equation}\label{poshomo}
        F(x,y,\alpha r, \alpha p, \alpha X) = \alpha F(x,y, r, p, X),
    \end{equation}
for all $\alpha \geq 0, x \in \Omega, y \in \T^N, p \in \R^N$ and $X \in S^N$.

Thus, for each $\epsilon \in (0,1)$, we have the existence of an eigenpair $(u^\epsilon, \lambda^\epsilon) \in C_+(\Omega) \times \R$ solving
    \begin{equation}\label{eq}
       \left \{ \begin{array}{rll} F(x, \nicefrac{x}{\epsilon}, u, Du, D^2 u) & = - \lambda^\epsilon u \quad & \mbox{in} \ \Omega, \\ u & = 0 \quad& \mbox{on} \ \partial \Omega. \end{array} \right .
    \end{equation}


Our first goal is to understand the homogenization behavior of problem~\eqref{eq} as $\epsilon \to 0$, and explore the possibility of having rates of convergence whenever homogenization holds. 
In fact, it is well-known (see Evans~\cite{E}) that for each $x, r, p, X$, there exists a unique $c \in \R$ such that
   	\begin{equation}\label{cell}
        F(x, y, r, p, D^2 v(y) + X) = c \quad \mbox{in} \ \T^N,
   	\end{equation}
has a viscosity solution. We can thus define $\bar F: \Omega \times \R \times \R^N \times S^N \to \R$ as
    $
    \bar F(x, r,p,X) = c,
    $
    where $c$ is the unique constant for which~\eqref{cell} has a solution. We call $\bar F$ the \textsl{effective Hamiltonian}. As it can be seen in~\cite{E}, $\bar F$ is continuous in all its arguments, uniformly elliptic and positively $1$-homogeneous, and therefore, the effective eigenvalue problem
        \begin{equation}\label{eqeff}
        \left \{ \begin{array}{rll} \bar F(x, u, Du, D^2 u) & = - \bar \lambda u \ & \mbox{in} \ \Omega, \\ u & = 0 \ & \mbox{on} \ \partial \Omega, \end{array} \right .
    \end{equation}
    is solvable by the (positive) eigenpair $(u, \bar \lambda) \in C_+(\Omega) \times \R$. As before, $\bar \lambda$ is unique and simple.

\medskip

The first main result of this paper is the following
\begin{teo}\label{teo1}
    Assuming $F$ satisfies~\eqref{elliptic},~\eqref{poshomo}. Let $(u^\epsilon, \lambda^\epsilon)$ be the solution of~\eqref{eq}. Then, we have $\lambda^\epsilon \to \bar \lambda$ as $\epsilon \to 0$. 
    
    Let  $u^\epsilon$ with $\| u^\epsilon\|_\infty = 1$ and $u$ solving~\eqref{eqeff} with $\| u \|_\infty = 1$, then $u^\epsilon \to u$ as $\epsilon \to 0$, uniformly on $\bar \Omega$.
\end{teo}

For linear operators, results of this type dates back to Osborn~\cite{O}, using the Spectral Theorem for compact operators in Banach spaces. In Kesavan~\cite{Kesavan}, the author obtains the result for self-adjoint linear problems in divergence form. From here, an ample class of problems have been addressed for operators having an appropriate weak formulation ($p$-Laplace operator), and the analysis is extended to the entire spectrum of the operator. Here we follow the viscosity approach initiated by Lions, Papanicoloau and Varadhan~\cite{LPV}, in junction with the tools of~\cite{BNV}.
The proof of this theorem is an adaptation of the perturbed test function method introduced in~\cite{E}. To some extent, the problem can be regarded as a multi-scale homogenization problem as in~\cite{ABM}, but the resonant nature of the problem prevents the use of comparison principles, which are replaced by the simplicity of the eigenvalue. We remark that operators like~\eqref{Isaacs} are commonly referred as \textsl{Bellman-Isaacs} operators, and naturally arise in some applications like the study of zero-sum stochastic differential games, see~\cite{FS}. In the case of a convex operator (namely, when $\mathcal A$ is a singleton), $F$ in~\eqref{Isaacs} is known as a \textsl{Bellman} operator, and appears in the study of stochastic optimal control problems.

\medskip

In the second part of the paper, we deal with rates of convergence for the principal eigenvalue problem.  In the case of \textsl{proper problems}, namely, for equations with the form
$$
F(x, \frac{x}{\epsilon}, u, Du, D^2 u) = 0,
$$
with $F$ such that $r \mapsto F(x, y, r, p, X)$ is nonincreasing, rates of convergence are at disposal in various contexts. In~\cite{CM}, Camilli and Marchi obtain a polynomial rate for Hamilton-Jacobi-Bellmann problems posed in $\R^N$. The convexity of the nonlinearity allows the use of $C^{2, \alpha}$ estimates to construct approximated correctors for the problem, which in turn lead to the rate $|u^\epsilon - u| = O(\epsilon^{\alpha'})$ for some $0 < \alpha' < \alpha$. Also based on regularity, higher order expansions of the solution $u^\epsilon$ are obtained by Kim and Lee~\cite{KL} for the Dirichlet problem on a bounded smooth domain. For the problem treated therein, the nonlinearity $F$ is convex, smooth and contains no lower-order terms. 
In~\cite{CS}, Caffarelli and Souganidis introduce the concept of $\delta$-viscosity solutions in order to tackle the non-convex case, and where comparison principle for this type of generalized solutions plays a key role. 


\medskip

Concerning the rate of convergence for the eigenvalues, we start focusing on the nonlinear eigenvalue problem
\begin{equation}\label{fully}
       \left \{ \begin{array}{rll} F \Big{(} \frac{x}{\epsilon}, D^2 u^\epsilon \Big{)} & = -\lambda^\epsilon u^\epsilon \quad & \mbox{in} \ \Omega, \\ u^\epsilon & = 0 \quad & \mbox{on} \ \partial \Omega, \end{array} \right .
    \end{equation}
and its associated effective problem
    \begin{equation}\label{efffully}
     \left \{ \begin{array}{rll}  \bar F (D^2 u) & = -\bar \lambda u \quad & \mbox{in} \ \Omega, \\ \bar u & = 0 \quad & \mbox{on} \ \partial \Omega. \end{array} \right .
    \end{equation}

Under convexity/regularity assumptions on $F$, we can prove the following rate of convergence for $\{ \lambda^\epsilon \}_\epsilon$.
    \begin{teo}\label{teo_rate_vp}
    Assume $F \in C^{4,1}(\T^N \times S^N)$ is convex in its second variable, and satisfies~\eqref{elliptic},~\eqref{poshomo}. Let $\lambda^\epsilon, \bar \lambda$ be, respectively, the principal eigenvalues associated to~\eqref{fully} and \eqref{efffully}. Then, there exists a constant $C$ just depending on $F$ and $\Omega$ such that 
        \[
        |\lambda^\epsilon - \bar \lambda| \leq C \epsilon.
        \]
    \end{teo}

The proof of Theorem~\ref{teo_rate_vp} is based on the Donsker-Varadhan variational characterization of the principal eigenvalue (see~\cite{A}), together with the correctors found in \cite{KL}, for which we too will need to impose some additional hypotheses on the operator $F$.  The same arguments allows us to prove the convergence of the principal eigenvalue for linear operators with lower order terms with the form
        \begin{align}\label{Leps}
        \begin{split}
            L^\epsilon u(x) 
            & = a^{ij}(x/\epsilon) \partial_{x_i x_j}^2 u(x) + b^j(x/\epsilon) \partial_{x_j} u(x) + c(x/\epsilon) u(x),
            \end{split}
        \end{align}
in the case the coefficients are smooth. For this, we need to adapt the method presented in~\cite{KL} in order to deal with the lower order terms in the construction of the correctors.

In fact, considering $L^\epsilon$ as above and the corresponding principal eigenvalue problem
        \begin{equation}\label{eqlin}
           \left \{ \begin{array}{rll} L^\epsilon u^\epsilon & = -\lambda^\epsilon u^\epsilon \quad & \mbox{in} \  \Omega, \\ u^\epsilon & = 0 \quad & \mbox{on} \  \partial \Omega, \end{array} \right .
        \end{equation}
we are able to provide a rate of convergence for the principal eigenfunctions associated to this problem. It is a well-kown fact that the effective problem inherits the main structure of the original one, see~\cite{E}. In particular, the effective operator related to $L^\epsilon$ is linear and takes the form
$$
\bar L u(x) = \bar a^{ij} \partial_{x_i x_j}^2 u(x) + \bar b^j \partial_{x_j} u(x) + \bar c u(x),
$$
for some constants entries $\bar a^{ij}, \bar b^j, \bar c$. Moreover, the effective matrix $\bar a$ is uniformly elliptic, with the same ellipticity constants of $(a^{ij})_{ij}$. Thus, the effective problem in this case takes the form
        \begin{equation}\label{eqefflin}
           \left \{ \begin{array}{rll} \bar L u & = - \bar \lambda u \quad & \mbox{in} \  \Omega, \\ u & = 0 \quad & \mbox{on} \  \partial \Omega. \end{array} \right .
        \end{equation}

Of course, if we ask $u > 0$ in $\Omega$ in~\eqref{eqefflin}, then $\bar \lambda$ is the principal eigenvalue associated to $\bar L$ and $\Omega$. Analogous comment for the perturbed problem~\eqref{eqlin}.

Concerning the convergence of the eigenfunctions (and eigenvalues), our result is the following
  \begin{teo}\label{teo_rate_function}
Let $L^\epsilon$ in~\eqref{Leps} with $a \in C^6(\T^N; S^N), b \in C^6(\T^N; \R^N)$, and $c \in C^6(\T^N)$. Let $(u, \bar \lambda)$ be a solution to~\eqref{eqefflin} with $u > 0$ in $\Omega$. Then, there exists $C > 0$ depending on the coefficients $a,b,c$ and $\Omega$, such that, for all $\epsilon \in (0,1)$, there exists $u^\epsilon$ solving~\eqref{eqlin} with $u^\epsilon > 0$ in $\Omega$, satisfying
		\begin{equation}
    		| \lambda^\epsilon - \bar \lambda| + \|u^\epsilon - u\|_{L^\infty(\Omega)} \leq C \epsilon.
		\end{equation}
\end{teo}


The rate of convergence for the principal eigenvalue is obtained in the same way as Theorem~\ref{teo_rate_vp}. For the rate of the eigenfunction, the result follows basically by Theorem 1 in~\cite{O}, but we provide a pure PDE proof here, by introducing an auxiliary homogenization problem whose solution serves as a pivot between $u^\epsilon$ and $u$. This strategy was previously used in~\cite{Kesavan}, and we combine it with the rates of convergence of~\cite{KL} for the auxiliary problem. Again, the simplicity of the principal eigenvalue  $\bar \lambda$ in~\eqref{eqefflin} plays a crucial role.
                                    
The paper is organized as follows: in Section~\ref{convergence} we prove the general convergence result, Theorem~\ref{teo1}. In Section~\ref{rate_value}, we analyze the convergence of the principal eigenvalue in some particular linear and nonlinear cases. Finally, in Section~\ref{rate_function} we prove rate of the convergence of the eigenfunction in the linear case. 
                                    
\section{General convergence result.}\label{convergence}

We start with the following result, see Lemma 3.1 in~\cite{E}.
    \begin{prop}\label{propbarH}
    	Under the assumptions of Theorem \ref{teo1}, for each $x, p \in \R^N, r \in \R, X \in \mathbb S^N$, there exists a unique constant $c = c(x, r, p, X)$ such that the problem~\eqref{cell} has a viscosity solution $v \in C^{1, \sigma}(\T^N)$. This solution is unique up to an additive constant.
    \end{prop}
    
As we mentioned in the introduction, this result allows us to define the effective Hamiltonian $\bar F \in C(\Omega \times \R \times \R^N \times S^N)$, which satisfies the assumption in~\cite{QS} to define the principal eigenvalue $\bar \lambda = \lambda_1^+(\bar F, \Omega)$ in the sense described in~\eqref{deflambdaFOmega}.

By the ellipticity assumption~\eqref{elliptic}, a simple modification of the computations of Lemma 1.1 in~\cite{BNV} involving the extremal Pucci operators, we have an uniform upper bound for the family of eigenvalues $\{ \lambda^\epsilon \}_\epsilon$ with the form
    $
        \lambda^\epsilon \leq C,
    $
where $C = C(r,\lambda, \Lambda, N, C_1)$ is a constant depending on the operator, the dimension, and $r \in (0,1)$, the radius of a ball contained in $\Omega$. On the other hand, it is easy to see, using the characterization of the principal eigenvalue, that $\lambda^\epsilon \geq -C_1$. Therefore, the sequence $\{ \lambda^\epsilon \}_\epsilon$ is equibounded.

\begin{proof}[Proof of Theorem \ref{teo1}]
%
%
We start by defining
    \begin{equation*}
 \lambda^* = \limsup_{\epsilon \to 0} \lambda_\epsilon, \quad \lambda_* = \liminf_{\epsilon \to 0} \lambda_\epsilon.
    \end{equation*}
 
 Replacing $F$ by $F - (C_1 + 1)u$, we can assume $F$ is proper and $\lambda_* > 0$.

For each $\epsilon\in (0,1)$, denote by $u^\epsilon$ the principal eigenfunction associated to $\lambda^\epsilon$ such that $\| u^\epsilon \|_\infty =1$. 

We consider $\epsilon_k \to 0$ such that $\lambda^k = \lambda^{\epsilon_k} \to \lambda^*$, and denote $u_k = u_{\epsilon_k}$. Given that the operator is elliptic we have by standard regularity theory that $\{ u_k \}_k$ is precompact in $C^\alpha(\bar \Omega)$ for some $\alpha > 0$. Then, up to subsequences, it converges uniformly to some $u \in C^{\alpha'}(\bar \Omega)$, $0\leq\alpha'<\alpha$, which is nonnegative in $\Omega$. Notice that there exists $\bar x \in \Omega$ where  $u(\bar x) = 1$ since the family $u_k$ is uniformly Lipschitz on the boundary of the domain, see Theorem 1.1 in~\cite{SS}. Thus, every sequence of points $x_k \in \Omega$ such that $u_k(x_k) = 1$ remains uniformly away the boundary, and we can take $\bar x$ as a limit point of the sequence. 

Now, we prove that the limit $u$ solves
\begin{equation}\label{nublado2}
\bar F(x, u, Du, D^2 u) \leq -\lambda^* u \quad \mbox{in} \ \Omega; \qquad u = 0 \quad \mbox{on} \ \partial \Omega,
\end{equation}
in the viscosity sense.

By contradiction, we assume this does not hold. Then, there exists $x_0 \in \Omega$ and a smooth function $\phi$ strictly touching $u$ from below at $x_0$ such that
\begin{align}\label{nublado}
\bar F(x_0, \phi(x_0), D\phi(x_0), D^2\phi(x_0)) \geq -\lambda^* \phi(x_0) + \theta.
\end{align}

Now, we consider the function $\phi_k(x) = \phi(x) + \epsilon_kv(\frac{x}{\epsilon_k}), \ x \in \Omega$, where $v$ is a solution to the cell problem~\eqref{cell} associated to $x = x_0$, $r = \phi(x_0)$, $p = D\phi(x_0)$ and $X = D^2 \phi(x_0)$. In view of~\eqref{nublado}, we claim that for all $r > 0$ small enough, $\phi^\epsilon$ satisfies
\begin{equation*}
F(x, \frac{x}{\epsilon_k}, \phi_k, D\phi_k, D^2\phi_k) \geq - \lambda^k \phi_k \quad \mbox{in} \ B_r(x_0),
\end{equation*}
in the viscosity sense.
This follows the usual pertubed test function method, but we provide the details for completeness. Let $x_1 \in B_r(x_0)$ and $\psi$ a smooth function touching $\phi_k$ from above. Then, we have $y_1 = x_1/\epsilon_k$ is a maximum point for the function
$$
y \mapsto v(y) - \frac{1}{\epsilon_k^2}(\psi(\epsilon_k y) - \phi(\epsilon_k y)).
$$

Notice that $v$ is $C^{1, \sigma}$ by classical results in elliptic theory (namely, Theorem VII.2 in~\cite{IL} together with Thereom 2.1 in~\cite{T}), while $\psi$ and $\phi$ are smooth. We thus compute
$$
\epsilon_k Dv(y_1) + D\phi(x_1) = D\psi(x_1).
$$

Using the equation solved by $v$, we have
\begin{equation*}
F(x_0, \frac{x_1}{\epsilon_k} , \phi(x_0), D\phi(x_0), D^2\psi(x_1) - D^2\phi(x_1) + D^2\phi(x_0)) \geq -\lambda^* \phi(x_0) + \theta.
\end{equation*}

Then, taking $k$ large enough and $r > 0$ small, but independent of $\psi$, by~\eqref{elliptic} we arrive at
\begin{equation*}
F(x_1, \frac{x_1}{\epsilon_k}, \phi_k(x_1), D\psi(x_1), D^2 \psi(x_1)) \geq -\lambda^* \phi(x_0) + \theta/2,
\end{equation*}
and using the definition of $\lambda^*$ and $\phi_k$, we conclude the desired viscosity inequality
\begin{equation*}
F(x_1, \frac{x_1}{\epsilon_k}, \phi_k(x_1), D\psi(x_1), D^2 \psi(x_1)) \geq -\lambda^k \phi_k(x_1).
\end{equation*}

Then, for all $r > 0$ small, we can use maximum principles in small domains (see Theorem 3.5 in~\cite{QS}) to conclude that
$$
\sup_{B_r(x_0)} \{ \phi_k - u_k \} \leq \sup_{\partial B_r(x_0)} \{ \phi_k - u_k \}.
$$
Passing to the limit as $k \to \infty$, and arranging terms, we conclude that
$$
\inf_{B_r(x_0)} \{ u - \phi \} \geq \inf_{\partial B_r(x_0)} \{ u - \phi \},
$$
which is a contradiction with the fact that $x_0$ is a strict test point. This concludes~\eqref{nublado2}. Notice that in particular we have $\bar F(x, u, Du, D^2 u) \leq 0$ in $\Omega$, from which we see that $u > 0$ in $\Omega$ by strong maximum principle in~\cite{BD}. Then, by definition of $\lambda(\bar F)$, we get
$$
\lambda^* \leq \lambda(\bar F).
$$
%
%
%

The same procedure leads to the existence of a function $v \geq 0$ in $\Omega$, not identically zero, viscosity solution to
\begin{equation*}
\bar F(x, v, Dv, D^2v) \geq -\lambda_* v \quad \mbox{in} \ \Omega, \quad v = 0 \quad \mbox{on} \ \partial \Omega.
\end{equation*}

Hence, if $\lambda_* < \lambda(\bar F)$, the operator $\bar F(u) + \lambda_* u$ satisfies the maximum principle. Thus,  $v \leq 0$ in $\Omega$, which is a contradiction. Then, 
$$
\lambda(\bar F) \leq \lambda_* \leq \lambda^* \leq \lambda(\bar F),
$$
and the convergence of the eigenvalues follows.

\medskip

Now we tackle the convergence of the eigenfunctions. By the previous result, we have every converging subsequence of $\{ u^\epsilon \}$ converges to a positive solution of the effective eigenvalue problem. By simplicity of the first eigenvalue, we conclude the result.
\end{proof}

\begin{remark}
It is possible to prove the convergence of the principal eigenfunctions with different normalizations. For instance, let $x_0 \in \Omega$ and consider the normalization $u^\epsilon(x_0) = 1$. Since $u^\epsilon$ solves
\begin{align*}
& M^-_{\lambda, \Lambda}(D^2 u) - C_1 |Du| - \tilde C_2 |u| \leq 0 \quad \mbox{in} \ \Omega \\
& M^+_{\lambda, \Lambda}(D^2 u) + C_1 |Du| + \tilde C_2 |u| \leq 0 \quad \mbox{in} \ \Omega,
\end{align*}
for some $\tilde C_2 > 0$, by Harnack inequality (c.f. Theorem 3.6 in~\cite{QS}), for each $K \subset \subset \Omega$ containing $x_0$, there exists a constant $C_K > 0$ depending on $K$ and the data, but not on $\epsilon$ such that $\sup_K u^\epsilon \leq C_K \inf_{K} u^\epsilon \leq C_K$. Then, denoting
$F^\epsilon(x,r,p,X) = F(x, \frac{x}{\epsilon}, r,p,X)$, we consider $K$ such that the operator $F^\epsilon + \lambda^\epsilon$ satisfies the maximum principle in $\Omega \setminus K$, see Theorem 3.5 in~\cite{QS}. This compact set $K$ does not depend on $\epsilon$. Then, using a suitable power of the distance to the boundary, we can construct a barrier function to conclude uniform bounds for the family $\{ u^\epsilon \}_\epsilon$ in $\Omega \setminus K$. Hence, the family is uniformly bounded in $\bar \Omega$ and we can perform the same proof above to conclude the convergence of $u^\epsilon$ to the principal eigenfunction $u$ with normalization $u(x_0) = 1$.
\end{remark}

The following corollary can be obtained from Theorem~\ref{teo1}, resembling a multiscale results in the spirit of Corollary 1 in~\cite{ABM}. It will be useful when we look for the rate of convergence for the first eigenfunction in Section \ref{rate_function}.
\begin{cor}\label{cor_resonant}
Assume $F$ satisfies the assumptions of Theorem~\ref{teo1}, and let $\{ f^\epsilon \}_\epsilon \subset C(\Omega)$ such that $f^\epsilon \to 0$ locally uniformly in $\Omega$ as $\epsilon \to 0$. Assume $u^\epsilon$ is a viscosity solution to the problem 
        \begin{equation*}\label{resonant}
         \left \{ \begin{array}{rll} F(x, \frac{x}{\epsilon}, u, Du, D^2 u) & = - \lambda^\epsilon u + f^\epsilon, \quad & \mbox{in} \ \Omega, \\
         u & = 0 \quad & \mbox{on} \ \partial \Omega. \end{array} \right .
        \end{equation*}
       with $\{ u^\epsilon\}_\epsilon$ uniformly bounded in $\Omega$. Then, up to subsequences, $u^\epsilon \to u$ uniformly in $\Omega$, with $u$ solving~\eqref{eqeff}.
\end{cor}


\section{Rate of convergence for the principal eigenvalue}\label{rate_value}
In this section we provide the proof of Theorem \ref{teo_rate_vp}. The proof involves the use of higher order correctors found in \cite{KL}, whose construction is highly nontrivial. In order to make the exposition clearer we present the proof first in the linear case. 

\subsection{Linear operator}
We concentrate on linear operators with the form \eqref{Leps}. 
For simplicity, given $a, M \in \R^{N \times N}$ we adopt the notation
    $$
    a M = \Tr \Big{(} a M \Big{)},
    $$
    and therefore, given a smooth function $\varphi: \R^N \to \R$, we have
    $$
    a D^2_{zz} \varphi(z) = \Tr(A D^2_{zz} \varphi(z)) = a^{ij} \partial^2_{z_i z_j} \varphi(z),
    $$
    where we have adopted the standard notation of sum over repeated indices. Whenever there is no room for confusion, we can write $\partial^2_{z_i z_j}$ as $\partial^2_{i j}$.

\medskip

Now we present the convergence result for the principal eigenvalues for linear operators.
\begin{prop}\label{lemalambdas}
 Assume the hypotheses of Thereom~\ref{teo_rate_function}. Let $\lambda^\epsilon$ be the principal eigenvalue for~\eqref{eqlin}, and $\bar \lambda$ the principal eigenvalue for~\eqref{eqefflin}. Then, there exists $C > 0$ just depending on $a,b,c$ and $\Omega$ such that $|\lambda^\epsilon - \bar \lambda| \leq C \epsilon$.
\end{prop}

\medskip

We prepare the proof of this lemma by introducing the problems involved in the construction of the correctors. Given $k,l=1,\ldots, N$, we consider the following cell problems
    \begin{equation}\label{defbara}
        a(y) D^2_{yy} \chi(y) + a_{kl}(y) = \gamma, \quad y \in \T^N,
    \end{equation}
where the pair $(\chi, \gamma)$ is the solution of the additive eigenvalue problem. As in~\cite{E}, the constant $\gamma$ is the unique constant for which this problem admits a solution, which we will denote by $\bar{a}_{kl}$. It is easy to see that these are the entries of the effective diffusion matrix in~\eqref{eqefflin}. The solution $\chi$, which is smooth and unique up to an additive constant, will be denoted by $\chi^{kl}$.

\medskip

Similarly, given $k=1, \ldots N$, $\bar b_k$ and $\bar c$ in~\eqref{eqefflin} are the unique ergodic constants associated to the problem
    \begin{equation}\label{defbarb}
        a(y) D^2_{yy} \eta(y) + b_{k}(y) = \bar b_{k}, \quad y \in \T^N,
    \end{equation}
whose solution is denoted by $\eta^k$, and
    \begin{equation}\label{defbarc}
        a(y) D^2_{yy} \nu(y) + c(y) = \bar c, \quad y \in \T^N,
    \end{equation}
whose solution is denoted simply by $\nu$. We adopt the normalization $\eta^k(y_0) = \nu(y_0)=0$ for some fixed $y_0 \in \T^N$, which implies that $\eta^k \equiv 0$ if $b\equiv0$ and $\nu\equiv 0$ if $c\equiv0$.

Now, for each $x \in \Omega, y \in \T^N$, we define our second-degree corrector as
    \begin{equation}\label{defw2}
        w_2(x,y) = \chi^{kl}(y) \partial^2_{kl}u(x) + \eta^k(y) \partial_{k}u(x) + \nu(y)u(x),
    \end{equation}
and notice that this function satisfies the equation
    \begin{align*}
         a(y) D_{yy}^2 w_2(x,y) + a(y) D_{xx}^2 u(x) + b(y) \cdot D_x u(x) + c(y) u(x) 
         = -\bar \lambda u(x).
    \end{align*}

With this in mind, we define additional correctors recursively as follows: given $k,l,m=1, \ldots, N$, we denote by $\bar a_{klm}$ the unique constant making the problem
    \begin{equation}\label{defaklm}
        a(y) D_{yy}^2 \chi(y) + 2 a_{*m}(y) \cdot D_y \chi^{kl}(y) = \bar a_{klm} \quad \mbox{in} \ \T^N,
    \end{equation}
solvable by a (unique up to an additive constant) function $\chi=\chi^{klm}$. Here we have denoted by $a_{*m}$ the $m$-th column of the matrix $a$.

To obtain the remaining correctors, we use the following equations in the torus
\begin{equation}\label{defbkl}
a(y) D_{yy}^2 \eta^{kl}(y) + 2 a_{*k}(y) \cdot D_y \eta^{l}(y) + b \cdot D_y \chi^{kl}(y) = \bar b_{kl} \quad \mbox{in} \ \T^N \quad (k,l=1,\ldots, N),
\end{equation}

\begin{equation}\label{defck}
    a(y) D_{yy}^2 \nu^k(y) + 2a_{*k}D_y \nu(y) + b \cdot D_y \eta^{k}(y) = \bar c_{k} \quad \mbox{in} \ \T^N \quad (k=1,\ldots, N),
\end{equation}

\begin{equation}\label{defd}
    a(y)D^2_{yy}\xi + b(y)\cdot D_y\nu(y) = \bar d \quad \mbox{in} \ \T^N.
\end{equation}

We adopt the normalization $\eta^{kl}(y_0)=\nu^k(y_0)=\xi(y_0)=0$ for some fixed $y_0\in \T^N$. Note that the ingredients in \eqref{defaklm}-\eqref{defd} are defined using only the original coefficients and the ``first round'' of correctors (i.e. those used to define $w_2$).

With the parameters defined in \eqref{defaklm}-\eqref{defd}, we define $\psi_1$ as the unique solution to
    \begin{equation}\label{defpsi1}
        \left \{ \begin{array}{l} \bar L \psi_1 = - \bar a_{klm} \partial^3_{k l m} u(x) - \bar b_{k l} \partial^2_{k l} u(x) - \bar c_k \partial_{k} u(x) - \bar d u(x) \quad \mbox{in}\ \Omega, \\ 
        \psi_1 = 0 \quad \mbox{on} \ \partial \Omega, \end{array} \right .
    \end{equation}
where we have assumed the standard index summation over repeated indices. 

In turn, we define the third-order corrector as
    \begin{align}
        w_3(x,y) =& \chi^{klm}(y) \partial^3_{klm} u(x) + \eta^{kl}(y)\partial^2_{kl} u(x) + \nu^k(y)\partial_{k} u (x) +\xi(y)u(x) + \nonumber\\
        & \quad +\chi^{kl}(y)\partial^2_{kl} \psi_1(x) + \eta^k(y)\partial_{k} \psi_1 (x) +\nu(y)\psi_1(x).\label{defw3}
     \end{align}
 
We claim that $w_3$ satisfies the equation
    \begin{align*}
        &a(y) D^2_{xx}\psi_1(x) + b(y)\cdot D_x\psi_1(x) + c(y)\psi_1(x) + 2a(y)D^2_{yx}w_2 + \\
        &\quad +b(y)\cdot D_y w_2 + a(y)D^2_{yy} w_3 = 0, \quad y \in \T^N.
    \end{align*}
Indeed, using \eqref{defw3} and \eqref{defw2}, together with \eqref{defaklm}-\eqref{defd}, we have
    \begin{align*}
        & a(y)D^2_{yy} w_3 \\
        ={}&  \left(\bar a_{klm} - 2 a_{*m}(y) \cdot D_y \chi^{kl}(y)\right) \partial^3_{klm} u  \\
        &\quad + \left(\bar b_{kl} - 2 a_{*k}(y) \cdot D_y \eta^{l}(y) + b \cdot D_y \chi^{kl}\bar b_{kl}\right)\partial^2_{kl}u\\
        &\quad + \left(\bar c_{k} - 2a_{*k}D_y \nu + b \cdot D_y \eta^{k}\right)\partial_k u +  \left(\bar d - b(y)\cdot D_y\nu(y)\right)u \\
        &\quad   + \left(\bar a_{kl} - a_{kl}(y)\right)\partial^2_{kl}\psi_1 + \left(\bar b_k - b_k(y)\right)\partial_k \psi_1 + \left(\bar c - c(y)\right)\psi_1\\
        ={}&  \bar a_{klm} \partial^3_{k l m} u + \bar b_{k l} \partial^2_{k l} u + \bar c_k \partial_{k} u + \bar d u\\
        &\quad - 2 a_{*m}(y) \cdot D_y \chi^{kl}(y) \partial^3_{klm} u - \left(2 a_{*k}(y)\cdot D_y \eta^{l}(y) + b \cdot D_y \chi^{kl}\bar b_{kl}\right)\partial^2_{kl}u\\
        &\quad - \left(2a_{*k}D_y \nu + b \cdot D_y \eta^{k}\right)\partial_k u - b(y)\cdot D_y\nu(y) u + \bar L\psi_1\\
        &\quad - a(y) D^2_{xx}\psi_1(x) - b(y)\cdot D_x\psi_1(x) - c(y)\psi_1(x),
    \end{align*}
while the definition \eqref{defw2} gives
    \begin{align*}
        2a(y)D^2_{yx}w_2 ={}& 2a_{ij}\left( \partial_{i}\chi^{kl}(y)\partial_j(\partial^2_{kl}u) + \partial_{i}\eta^{k}(y)\partial_j(\partial_{k}u) + \partial_{i}\nu(y)\partial_j u\right)\\
        ={}& 2 a_{*m}(y) \cdot D_y \chi^{kl}(y)\partial^3_{klm} u + 2 a_{*k}(y) \cdot D_y \eta^{l}(y)\partial^2_{kl} u\\
        &\quad + 2a_{*k}\cdot D_y \nu\partial^3_{k} u,
    \end{align*}
and
    \begin{align*}
        b(y) \cdot D_y w_2 = b(y)\cdot D_y \chi^{kl}(y) \partial^2_{kl} u + b(y)\cdot D_y \eta^{k}(y) \partial_{k} u + b(y)\cdot D_y \nu(y) u.
    \end{align*}
    
The claim follows by summing-up the previous equalities and use~\eqref{defpsi1}.

For $i = 2,3$, we consider the boundary correctors $z_k^\epsilon$ as the unique solution to the Dirichlet problem
    \begin{equation}
        L^\epsilon z_k^\epsilon = 0 \quad \mbox{in} \ \Omega; \qquad z_k^\epsilon(x) = - w_k(x, x/\epsilon) \quad \mbox{on} \ \Omega.
    \end{equation}
  
Finally, we define the full corrector
    \begin{align}\label{corrector}
        v^\epsilon(x) = \epsilon \psi_1(x) + \epsilon^2(w_2(x, x/\epsilon) + z_2^\epsilon(x)) + \epsilon^3(w_3(x, x/\epsilon) + z_3^\epsilon(x))
    \end{align}
    
For $k=1,2,3$, by the results of~\cite{KL} we have $w_k, z_k^\epsilon$ are uniformly bounded in terms of $\epsilon$, and that $w_k$ has $C^{2,\alpha}$ estimates in $y$, and $C^{5-k, \alpha}$ estimates in $x$, uniform with respect to $\epsilon$. In particular, we get $\|v^\epsilon\|_\infty \leq C \epsilon$ for some $C > 0$.

Now we are in position to provide the

\begin{proof}[Proof of Proposition~\ref{lemalambdas}] 

Using the classical minmax formula for the principal eigenvalue (see~\cite{DV}), there exists a probability measure on $\Omega$, denoted by $\mu^\epsilon$, such that
\begin{equation*}
-\lambda^\epsilon = \inf_{\phi > 0} \int_{\Omega} \frac{L_\epsilon \phi}{\phi}(x) d\mu^\epsilon(x),
\end{equation*} 
where the infimum is taken among all smooth functions $\phi$ which are positive in $\Omega$.
Now, recalling $v^\epsilon$ in~\eqref{corrector}, consider the function
$$
\bar v^\epsilon = u + v^\epsilon + C_1 \epsilon + C_2 \epsilon d^{\gamma}, 
$$
where $\gamma \in (0,1)$ and $C_1, C_2 > 0$ can be taken in such a way $\bar v^\epsilon > 0$ in $\Omega$. Then, by the estimates of Theorem 1.1 in~\cite{KL}, we have
\begin{align*}
-\lambda^\epsilon \leq & \int_{\Omega} \frac{L_\epsilon \phi}{\phi}(x) d\mu^\epsilon(x) \\
\leq & \int_{\Omega} \frac{-\bar \lambda u(x) - c C_2 \epsilon d^{\gamma - 2}(x) + C\epsilon^3 }{\bar v^\epsilon(x)} d\mu^\epsilon(x) \\
= & -\bar \lambda + \int_{\Omega} \frac{-\bar \lambda (u(x) - \tilde v^\epsilon(x)) - c C_2 \epsilon d^{\gamma - 2}(x) + C\epsilon^3 }{\bar v^\epsilon(x)} d\mu^\epsilon(x).
\end{align*}

At this point, since $\gamma < 1$, we can find $\delta > 0$ independent of $\epsilon$ such that the numerator in the integral term above is nonpositive for $x \in \Omega$ such that $d(x) < \delta$. On the other hand we see that $\bar v^\epsilon(x) \geq c_\delta > 0$  if $d(x) \geq \delta$. Hence, we get that
\begin{equation*}
-\lambda^\epsilon \leq -\bar \lambda + C \epsilon c_\delta \int_{\{ x : d(x) \geq \delta\}} d\mu^\epsilon(x),
\end{equation*}
 and from this we conclude that $\lambda^\epsilon \geq \bar \lambda - C\epsilon$.

 \medskip
 
 Now we prove an upper bound for $\lambda^\epsilon$. For this, we consider the function
 $$
 \underline v^\epsilon = u + v^\epsilon - C_1 \epsilon - \epsilon d^{\gamma},
 $$
 for some $C_1 > 0$ such that $\underline v^\epsilon < 0$ on $\partial \Omega$ and $\underline v^\epsilon(x_0) > 1/2$ if $\epsilon$ is small enough. Then, we see that for each $x \in \Omega$ we have
 $$
L^\epsilon \underline v^\epsilon = - \bar \lambda u(x) + \epsilon a D^2_{xx} d^{\gamma}(x) - O(\epsilon^3).
 $$
 
 We notice that for some $\delta_0$ just depending on the smoothness of the domain, there exists some $c > 0$ such that $a D_{xx}^2 d^\gamma(x) \geq c d^{\gamma - 2}(x)$ for $d(x) \leq \delta_0$. Moreover, for each $\delta > 0$, we have $a D_{xx}^2 d\gamma(x) \geq -C_\delta$ if $d(x) \geq \delta$. 
 
 It is easy to see that
 \begin{align*}
L^\epsilon \underline v^\epsilon \geq & -\bar \lambda \underline v^\epsilon + \epsilon a D^2_{xx} d^{\gamma}(x) - \bar C\epsilon \\
\geq &  -(\bar \lambda + C_3 \epsilon) \underline v^\epsilon + \epsilon(C_3 \underline v^\epsilon + a D^2_{xx} d^{\gamma}(x) - \bar C)
 \end{align*}
in $\Omega$, for some universal constant $\bar C$.

Write $h = C_3 \underline v^\epsilon + a D^2_{xx} d^{\gamma}(x) - \bar C$ and let $0 < \delta \leq \delta_0$. We divide the analysis in different regions on $\Omega$. Since $u > 0$ in $\Omega$, for each $\delta > 0$ and all $\epsilon$ small enough in terms of $\delta$, we have
\begin{align*}
h \geq C_3 c_\delta - C_\delta - \bar C
\end{align*}
in the region $d(x) \geq \delta$. Thus, we fix $\delta$ and enlarge $C_3$ in order to get $h \geq 0$ in this region. 

Now, if $R \epsilon \leq d(x) \leq \delta$, by Hopf's Lemma over $u$, we can take $R$ large enough to get $\tilde v^\epsilon \geq 0$ in this region, and therefore
$$
h \geq c \delta^{\gamma - 2} - \bar C,
$$
and dimishing $\delta$ just in terms of the universal constants $c, \bar C$ we get $h \geq 0$ in this region too. 

Fix $\delta$ and $C_3$ as above, if $d(x) \leq R \epsilon$, we have $\tilde v^\epsilon \leq C_R \epsilon$ for some $C_R > 0$. Then for $x$ such that $d(x) \leq \epsilon$, we have $\underline v^\epsilon \leq C\epsilon$ and
\begin{equation*}
h \geq - C_3 C_R \epsilon + c \epsilon^{\gamma - 2} - \bar C,
\end{equation*}
and then, just taking $\epsilon$ small enough we conclude that $h \geq 0$ here. 

Therefore, we have 
$$
L_\epsilon \underline v^\epsilon \geq -(\bar \lambda + C_3 \epsilon) \underline v^\epsilon \quad \mbox{in} \ \Omega.
$$

Multiplying $\underline v^\epsilon$ by a positive constant if necessary, we have that $\underline v^\epsilon$ touches $u^\epsilon$ from below at some interior point $x \in \Omega$, and therefore
$$
-(\bar \lambda + C_3 \epsilon) \underline v^\epsilon(x) \leq L^\epsilon \underline v^\epsilon(x) \leq \lambda^\epsilon \underline v^\epsilon(x),
$$
from which $\lambda^\epsilon \leq \bar \lambda + C_3 \epsilon$, which concludes the proof.
\end{proof}

\subsection{Nonlinear operator}
We are going to consider here the fully nonlinear homogenization problem~\eqref{fully},
where $F$ satisfies the assumptions of Theorem~\ref{teo_rate_vp}. 
For $M \in S^{N}$, $\bar F(M)$ is defined as the unique constant $c$ such that the problem
    $$
        F(y, M + D_{yy}^2 w(y)) = c, \quad y \in \T^N,
    $$
admits a viscosity solution $w = w(y; M)$. By classical Evans-Krylov estimates for convex, uniformly elliptic operators (see \cite{CC, E2, Kr}) we have $y \mapsto w(y;M)$ is $C^{2, \alpha}(\T^N).$

By standard results in periodic homogenization, we have $\bar F$ is convex, positively $1$-homogeneous, uniformly elliptic (with the same ellipticity constants of $F$), see~\cite{E}. In addition, if $F \in C^{4,1}$, then $\bar F$ too, see Proposition 3.2.1 in~\cite{KL}. By Theorem~\ref{teo1}, we have the principal eigenvalue $\lambda^\epsilon$ in~\eqref{fully} converges to the principal eigenvalue $\bar \lambda$ in~\eqref{efffully}. 
Again by Evans-Krylov estimates, we have the principal eigenfunction $u$ of this problem is in $C^{6}$ since $F \in C^{4,1}$, see Lemma 3.3.1 in~\cite{KL}.

\medskip    

In order to prove Theorem \ref{teo_rate_vp}, we introduce the appropriate corrector terms. As usual, our second-order term in the expansion is given by
$$
w_2(x,y) := w(y; D^2_{xx} u(x)), \quad x \in \Omega, \ y \in \T^N,
$$
and therefore, for each $x \in \Omega$ and by the definition of $w(y;M)$, we have
\begin{equation}\label{w2F}
F(y, D_{xx}^2u(x) + D^2_{yy}w_2(x,y)) = -\bar \lambda u(x), \quad y \in \T^N.
\end{equation}

For $x \in \Omega$, $y \in \T^N$ denote
$$
X^0 = D_{xx}^2 u(x) + D_{yy}^2 w(y, D^2_{xx} u(x)), \quad a_{ij}(x,y) = D_{p_{ij}} F(y, X^0).
$$

The entries $a_{ij}$ are the coefficients of the linearization of the equation around the leading profile $(x,y) \mapsto u(x) + \epsilon^2 w_2(x, x/\epsilon)$. In fact, we have the existence of functions $\{ w_k \}_{k=1}^4$ with $w_k$ sufficiently smooth, and such that the function
\begin{equation}\label{weps}
w^\epsilon(x) = u(x) + \epsilon w_1(x) + \epsilon^2 w_2(x, x/\epsilon) + \epsilon^3 w_3(x, x/\epsilon) + \epsilon^4 w_4(x, x/\epsilon),  
\end{equation}
satisfies
\begin{equation}\label{18}
F(\frac{x}{\epsilon}, D^2_{xx} w^\epsilon) = -\bar \lambda u(x) + O(\epsilon) \quad \mbox{in} \ \Omega.
\end{equation}

We refer to Theorem 1.2.1 in \cite{KL} for the proof of the expression~\eqref{18}. The construction of the correctors follows an inductive argument that we describe next.

We start noticing that first term in the expansion does not depend on $y$. It is constructed in the following way: for each $x$, we have the existence of a unique constant $\Psi = \Psi_1(x)$ for which the ergodic problem
\begin{equation*}
a_{ij}(x,y) D^2_{y_i y_j} v(x,y) + 2 a_{ij}(x,y) D^2_{x_i y_j} w_2(x,y) = \Psi
\end{equation*}
has a solution. 

The entries of the linearized effective problem correspond to 
$$
\bar a_{ij} = D_{p_{ij}} \bar F(D^2 u(x)),
$$
and are such that $\bar a$ is uniformly elliptic. Then, we can solve the Dirichlet problem
\begin{equation*}
\bar a_{ij} D_{x_i x_j}^2 \psi = - \Psi(x), \ \mbox{in} \ \Omega; \qquad \psi = 0 \ \mbox{on} \ \partial \Omega,
\end{equation*}
and from here we find $w_1 = \psi_1$ in the expansion of $w^\epsilon$. 

In fact, with $w_1$ and $w_2$ at hand, we can run the inductive argument in Lemma 3.3.2 in~\cite{KL} to obtain $w_3$ and $w_4$ in the expansion of $w^\epsilon$ in~\eqref{weps} (with $\psi_1 = \psi$ and $\Psi_1 = \Psi$ here).

%
%
%
%

With the above elements, we are able to provide the

\begin{proof}[Proof of Theorem~\ref{teo_rate_vp}:]
Using the Donsker-Varadhan characterization of the principal eigenvalue, we have
\begin{equation*}
-\lambda^\epsilon = \inf_{\phi > 0} \int_{\Omega} \frac{F(x/\epsilon, D^2 \phi(x))}{\phi(x)} d\mu^\epsilon(x),
\end{equation*} 
for some probability measure $\mu^\epsilon \in \mathcal P(\bar \Omega)$.

Now, since $w_i$ is uniformly bounded for $i=1,...,4$, we have $w^\epsilon(x) > 0$ for each $x \in \Omega$, $d(x) \leq c \epsilon$ for some $c > 0$ small enough, but independent of $\epsilon$. Thus, the function
$$
\phi(x) = w^\epsilon(x) + \epsilon + C\epsilon d^{\gamma}(x), \quad x \in \Omega,
$$
is strictly positive for some $C > 0$ large enough, and $\phi \in C^2(\Omega)$. Then, we conclude that for each $x \in \Omega$ and $\epsilon > 0$ we have
$$
F(\frac{x}{\epsilon}, D^2\phi(x)) = F(\frac{x}{\epsilon}, D^2 w^\epsilon(x)) + C \mathcal M^+ (D^2 d^\gamma(x)),   
$$
where $\mathcal M^+$ is the extremal Pucci operator with ellipticity constants $0 < \lambda_F, \Lambda_F < +\infty$. Using this function $\phi$ in the formula of the eigenvalue, we have
\begin{align*}
-\lambda^\epsilon \leq & \int_{\Omega} \frac{-\bar \lambda u(x) + O(\epsilon) - C c \epsilon \lambda_F d^{\gamma - 2}(x)}{u(x) + \epsilon + C \epsilon d^\gamma(x)} d\mu^\epsilon(x) \\
\leq & -\bar \lambda + \int_{\Omega} \frac{O(\epsilon) - C c \epsilon \lambda_F d^{\gamma - 2}(x)}{u(x) + \epsilon + C \epsilon d^\gamma(x)} d\mu^\epsilon(x),
\end{align*}
and from here, if we consider $\delta_0 > 0$ small but independent of $\epsilon$, we have the numerator inside the last integral is nonpositive for $d(x) \leq \delta_0$. Hence we get
$$
-\lambda_\epsilon \leq -\bar \lambda + \int_{d(x) \geq \delta_0} \frac{C\epsilon}{u(x) + \epsilon + C \epsilon d^\gamma(x)} d\mu^\epsilon(x) \leq -\bar \lambda + C\epsilon,
$$
from which, we get the lower bound 
$$
\lambda^\epsilon \geq \bar \lambda - C \epsilon,
$$
for some $C > 0$.

\medskip

For the upper bound, we consider the function
\begin{equation}
\phi = w^\epsilon - \epsilon - C \epsilon d^{\gamma}(x), \quad x \in \Omega.
\end{equation}

We can take $C > 0$ such that $\phi > 0$ for each $d(x) > \delta_0$ for some $\delta_0$ small, and  $\phi(x) < 0$ for all $d(x) \leq \delta_1$ with $\delta_1 < \delta_0$.

Then, for each $x \in \Omega$, we have
\begin{align*}
F(\frac{x}{\epsilon}, D^2 \phi(x)) \geq  & -\bar \lambda u(x) - \tilde C\epsilon + C \epsilon \mathcal M^- (-D^2 d^\gamma(x))  \\
\geq & -\bar \lambda \phi(x) - \tilde C\epsilon + C \lambda_F \epsilon d^{\gamma - 2}(x) \\
= & -(\bar \lambda + \ell \epsilon) \phi(x) + \ell\epsilon \phi(x) - \tilde C\epsilon + C \lambda_F \epsilon d^{\gamma - 2}(x),
\end{align*}
for some $\ell > 0$. Notice that for $d(x) \leq \delta$ with $\delta$ small, but independent of $\epsilon$, we have $\tilde C + \lambda_F d^{\gamma - 2}(x) \geq 0$. For $d(x) \geq \delta$ we have $d^{\gamma - 2(x)}$ is bounded and therefore we conclude that
\begin{equation*}
F(\frac{x}{\epsilon}, D^2 \phi(x)) \geq -(\bar \lambda + \ell \epsilon )\phi \quad \mbox{in} \ \Omega,
\end{equation*}
by taking $\ell$ suitably large.

Since $\phi(x) < 0$ for each $x$ a neighborhood of the boundary, up to a multiplicative constant, we have $\phi$ touches from below $u^\epsilon$ at a point $x_0 \in \Omega$, with $\phi(x_0) = u^{\epsilon}(x_0) > 0$, and  we have
$$
-(\bar \lambda +\ell \epsilon) \phi(x_0) \leq F(x_0/\epsilon, D^2 \phi(x_0)) \leq -\lambda^\epsilon u^\epsilon(x_0),
$$
and from here we get 
$
\lambda^\epsilon \leq \bar \lambda + \ell \epsilon.
$ This concludes the proof.
\end{proof}

We believe that the discussion presented here allows to get Proposition~\ref{teo_rate_vp} for more general smooth, convex, $1$-homogeneous nonlinearity $F=F(x,y,r,p,X)$. More specifically, the method used to find the correctors associated to the lower order terms in the first part of this section, can be combined with the intricate linearization procedure in~\cite{KL}, but we do not pursue in this direction. 

\section{Proof of Theorem~\ref{teo_rate_function}}\label{rate_function}

This section is entirely devoted to the 
\begin{proof}[Proof of Theorem~\ref{teo_rate_function}:] We have proved the reate of convergence for the eigenvalues in Proposition~\ref{lemalambdas}.

By replacing $\bar \lambda$ and $\lambda^\epsilon$ with $\bar \lambda + C_1 + 1$ and $ \lambda^\epsilon + C_1 + 1$, respectively, we can assume that $L^\epsilon, \bar \lambda$ are proper operators.
Let $u$ be the solution to~\eqref{eqefflin} with $\| u \|_\infty = 1$. Given $\epsilon\in (0,1)$ consider the following homogenization problem
\begin{align}\label{w_e}
\left \{ \begin{array}{rll}
L^\epsilon w^\epsilon & = -\bar\lambda u \quad & \mbox{in} \ \Omega \\
w^\epsilon& = 0 \quad & \mbox{in} \ \partial\Omega,
\end{array} \right .
\end{align}

By the results in~\cite{KL}, we have the existence of $C > 0$ just depending on the coefficients of $L$ and $\Omega$ (but not on $\epsilon$), such that 
\begin{equation}\label{weu}
\| w^\epsilon - u\|_\infty \leq C \epsilon.
\end{equation}

Assume $u^\epsilon$ is a normalized eigenfunction for~\eqref{eqlin} (say $\| u^\epsilon \|_\infty = 1$), and consider $z^\epsilon = u^\epsilon - w^\epsilon + t_\epsilon u^\epsilon$ for some $t_\epsilon \in \R$ to be fixed. It is easy to see that $z^\epsilon$ solves
the problem
\begin{align}\label{z_e}
\left \{ \begin{array}{rll} 
L^\epsilon z^\epsilon + \lambda^\epsilon z^\epsilon & = -\lambda^\epsilon w^\epsilon+\bar\lambda u \quad  & \mbox{in} \ \Omega, \\
z^\epsilon & = 0 \quad & \mbox{on} \ \partial \Omega. \end{array} \right .
\end{align}

We consider $t_\epsilon$ such that 
\begin{equation}\label{z_ort_u}
    (z^\epsilon, u^\epsilon)  = 0,
\end{equation}
where $(\cdot,\cdot)$ denotes the inner product in $L^2(\Omega)$. This is possible by taking 
$$
t_\epsilon = \frac{(w^\epsilon, u^\epsilon)}{\| u^\epsilon \|_{L^2}} -1.
$$

Notice that the family $\{ z^\epsilon \}_\epsilon$ is uniformly bounded, and by standard elliptic estimates we have it is equicontinuous in $\bar \Omega$.



With this choice, we claim the existence of $C_0 > 0$ such that
\begin{align}\label{claim}
\|z^\epsilon\|_\infty \leq C_0 \epsilon,
\end{align}
for all $\epsilon \in (0,1)$. Let us argue by contradiction, assuming the existence of a sequence $\epsilon_n \in (0,1)$ such that $z^n=z^{\epsilon_n}$ satisfy
\[
\|z^n\|_\infty \geq n \epsilon_n,
\]
as $n \to \infty$. Notice that, in particular, we have $z^n$ is not identically zero in $\Omega$ for all $n$, and that $\epsilon_n \to 0$ as $n \to \infty$. Since $\|\lambda^{\epsilon_n} w^{\epsilon_n} - \bar \lambda u\|_\infty \leq C \epsilon_n$ for some $C > 0$, we have $\hat z^{n} = z^{n}/\| z^n \|_\infty$ solves
\begin{align*}
\left \{ \begin{array}{rll} 
L^{\epsilon_n} \hat z^n + \lambda^{\epsilon_n} \hat z^n & = o_n(1) \quad  & \mbox{in} \ \Omega, \\
\hat z^n & = 0 \quad & \mbox{on} \ \partial \Omega, \end{array} \right .
\end{align*}
where $o_n(1) \to 0$ as $n \to \infty$. Then, substracting a subsequence (that we still denote by $\hat z^n$), by Corollary \ref{cor_resonant} we have that $(\hat z^n)$ converges to a nontrivial solution $z$ to the eigenvalue problem
    \begin{align*}
        \left \{ \begin{array}{rll}
            \bar L z   & = -\bar\lambda z \quad & \mbox{in} \ \Omega \\
            z& = 0 \quad & \mbox{in} \ \partial\Omega,
        \end{array} \right .
    \end{align*}
hence $z$ is an eigenfunction of $L$ with $\|z\|_\infty=1$. On the other hand, using~\eqref{z_ort_u} we have
\[
(z,u)=0,
\]
but this implies that $z=0$, a contradiction, and the claim is proven.

From~\eqref{claim}, we get that
$$
\| w^\epsilon - (1 + t_\epsilon) u^\epsilon \|_\infty \leq C_0 \epsilon,
$$
and from here we get the result, by triangle inequality,~\eqref{weu} and replacing $u^\epsilon$ by $(1 + t_\epsilon) u^\epsilon$ in the statement of the theorem.
%
\end{proof} 

%
%


\bigskip

\noindent {\bf Acknowledgements.} 
 A.~R.-P. was partially supported by Fondecyt Grant Postdoctorado Nacional No.~ 3190858. G. D. was partially supported by Fondecyt Grant 1190209. E.~T. was partially supported by Fondecyt Grant No.~1201897.

\end{document}